\documentclass[12pt]{gtpart}
\usepackage{amssymb,amsfonts,amsthm,amsmath,amscd}
\usepackage{url}

\title[Relative property A]{Relative property A and relative amenability for countable groups}

\author{Ronghui Ji}
\address{
		Department of Mathematical Sciences\\
		Indiana University-Purdue University, Indianapolis\\ 
		402 N. Blackford Street\\
		Indianapolis, IN 46202\\
		}
\email{ronji@math.iupui.edu}

\author{Crichton Ogle}
\address{
		Department of Mathematics\\
		The Ohio State University\\
        Columbus, OH 43210\\
        }
\email{ogle@math.ohio-state.edu}

\author{Bobby W. Ramsey}
\address{
		Department of Mathematics\\
		The Ohio State University\\
        Columbus, OH 43210\\
        }

\email{ramsey.313@math.osu.edu}

\newtheorem*{thm*}{Theorem}
\newtheorem*{cor*}{Corollary}
\newtheorem{thm}{Theorem}[section]
\newtheorem{cor}[thm]{Corollary}
\newtheorem{prop}[thm]{Proposition}
\newtheorem{lem}[thm]{Lemma}

\theoremstyle{definition}
\newtheorem*{rmk*}{Remark}
\newtheorem{defn}[thm]{Definition}

\newcommand{\N}{{\mathbb{N}}}
\newcommand{\BG}{{\beta G}}
\newcommand{\CG}{{\mathbb{C}G}}
\newcommand{\isom}{{\,\cong\,}}
\newcommand{\E}{\mathcal{E}}
\newcommand{\V}{\mathcal{V}}
\newcommand{\im}{\operatorname{im}}
\newcommand{\mH}{\mathcal{H}}
\newcommand{\I}{\mathcal{I}}
\newcommand{\Ext}{\operatorname{Ext}}
\newcommand{\Hom}{\operatorname{Hom}}
\newcommand{\bExt}{\operatorname{bExt}}
\newcommand{\bHom}{\operatorname{bHom}}
\newcommand{\Prob}{\operatorname{Prob}}
\newcommand{\normal}{\triangleleft}

\allowdisplaybreaks[4]

\begin{document}

\begin{abstract}
We define a relative property A for a countable group with respect to
a finite family of subgroups. Many characterizations for relative property A  are given. In particular a relative bounded cohomological 
characterization shows that  if $G$ has property A relative to a family of subgroups $\mH$ and if each $H\in \mH$ has property A, 
then $G$ has property A. This result leads to new classes of groups that have property A. In particular, groups are of property A 
if they act cocompactly on locally finite property A spaces of  bounded geometry with any stabilizer of property A. Specializing the 
definition of relative property A, an analogue definition of relative amenability for discrete groups are introduced and similar results are obtained.
\end{abstract}   

\maketitle

\tableofcontents

\section{Introduction}
Property A is a geometric condition on metric spaces introduced by Yu \cite{yu2} for studying 
coarse embeddability into Hilbert spaces.  It amounts to a non-equivariant generalization of amenability.  Among many important consequences of property A are the validity of the Coarse Baum-Connes Conjecture (which, in turn, implies the Strong Novikov Conjecture), the Gromov-Lawson-Rosenberg Conjecture on the existence of positive scalar curvature, and Gromov's zero-in-the-spectrum conjecture. It is the work of Higson and Roe \cite{HR} that characterized property A in terms of analytical conditions, which are similar to classical conditions characterizing amenable groups. Since amenable groups are of property A, and they can also be characterized by vanishing of certain bounded cohomology groups, Higson proposed to characterize property A groups by cohomological conditions. Several related results in this direction have recently appeared in \cite{BNNW}, \cite{BNW}, and \cite{DN}. In \cite{BNW}  Brodzki, Niblo, and Wright give a characterization of property A for metric spaces
via the construction of an asymptotically invariant cohomology theory, while in \cite{DN} Douglas and Nowak give a partial characterization of property A in terms of bounded cohomology. The program in \cite{DN} was completed by Brodzki, Niblo, Nowak, and Wright in \cite{BNNW} using bounded cohomology with certain coefficients.  In particular, they show that 
$G$ has property A if and only if for every appropriate coefficient module $\E$, $H^q_{b}(G;\E^*) = 0$ 
for $q \geq 1$.  Monod has obtained a similar characterization in \cite{Mo2}, where property A is further shown to be
equivalent to the relative injectivity of a class of Banach modules associated to the group.

In the current paper, we expand upon the methods of \cite{BNNW} by examining the relative bounded cohomology of the 
countable discrete group $G$ with respect to a finite family of 
subgroups, $\mH$, denoted by $H^*_{b}(G, \mH; \V)$.  For each choice of coefficient module, there is
a long-exact sequence relating $H^*_{b}(G;\V)$, $H^*_{b}(G,\mH;\V)$, and the bounded cohomology of
the subgroups \cite{JOR}.

We then define relative property A for $G$ equipped with a left-invariant proper metric
relative to a finite family of subgroups $\mH$.  This is given as a relative version of Yu's original
definition.  Several equivalent formulations of relative property A are developed, including analogues of
amenable actions on a compact space and Reiter's condition.  Most importantly, $G$ has relative property A with respect
to $\mH$, if and only if for any $\ell^1$-geometric $G$-$C(X)$ module $\E$, the map
$H^0_{b}(\mH;\E^*) \to  H^1_{b}(G,\mH;\E^*)$ is surjective.  In conjunction with the results of \cite{BNNW}, 
we arrive at the following theorem.
\begin{thm*}
Suppose that the countable discrete group $G$ has relative property A with respect to the finite family of subgroups $\mH$.
Then $G$ has property A if and only if each of the subgroups $H \in \mH$ has property A.
\end{thm*}
In particular, this generalizes the theorems of Ozawa \cite{Oz} and Dadarlat-Guentner \cite{DG} for relatively hyperbolic groups,
and is similar in spirit to Osin's results on finite asymptotic dimension \cite{Os}.

The following theorem describes conditions that ensure relative property A.

\begin{thm*}
Suppose that a finitely generated discrete group $G$ acts cocompactly on a uniformly discrete metric space
with bounded geometry, $(X, d_X)$.  Fix a family of representatives $\mathcal{R}$ of orbits of $G$ in
$X$, and let $\mH$ be the family of subgroups occurring as the stabilizers of elements of $\mathcal{R}$.
If $(X,d_X)$ has property A, then $G$ has relative property A with respect to $\mH$.
\end{thm*}

A corollary to the theorem is the following result on a complex of groups.
\begin{cor*}
Suppose the finitely generated discrete group $G$ is the fundamental group of a developable finite dimensional complex of groups
whose development is a locally finite complex with property A.  Then $G$ has relative property A with
respect to the vertex groups.  In particular, if each vertex group has property A, then $G$ has property A. 
\end{cor*}

This also gives an extension of a Theorem of \cite{BCGNW}, where it is shown that a finite dimensional
CAT($0$) cube complex has property A.  We have the following corollary which generalizes a theorem of Bell \cite{B} and Guentner-Tessera-Yu \cite{GTY} for property A groups.
\begin{cor*}
Suppose that the finitely generated group $G$ acts cocompactly on a locally finite, finite dimensional CAT($0$) cube complex.
If any vertex stabilizer has property A, then $G$ has property A.
\end{cor*}

By specializing the definition of relative property A, we formulate a notion of relative amenability.
A countable discrete group $G$ is relatively amenable to a finite family of subgroups $\mH$ if there is a
$G$-invariant mean on $\ell^{\infty}(G/\mH)$, where $G/\mH$ is the disjoint of all $G/H$ for $H \in \mH$.
Our characterization of relative property A gives a cohomological characterization of relative amenability
as well.  This allows one to determine amenability for a group in terms of group actions on metric spaces, via a
corollary to Theorem \ref{thm:RelAmenActionOnSpace} which states the following:
\begin{cor*}
Suppose the countable discrete group $G$ acts cocompactly on a uniformly discrete metric space with bounded geometry.
If the space is amenable in the sense of Block and Weinberger \cite{BW}, and any point stabilizer is amenable, then $G$ is amenable.
\end{cor*}

We anticipate that this will provide many new examples of amenable groups. 

We would also like to point out that from the works of Guentner-Kaminker \cite{GK} and Ozawa \cite{Oz2} Property A for countable discrete groups is equivalent to the exactness  of the group \cite{KW}. We will discuss a relative version of the exactness in a future publication.

We now outline what will follow.

In section 2 we recall the definition of property A for uniformly discrete metric spaces of bounded geometry.
We then review the bounded cohomology interpretation of property A given in \cite{BNNW}.  
In section 3 we review relative bounded cohomology and the Bieri-Eckmann type long-exact sequence relating 
relative and absolute bounded cohomologies \cite{JOR}.

In section 4 we introduce the notion of relative property A.  This is based on the existence of a countable set $K$
on which the group acts cofinitely, with the elements of $\mH$ and their conjugates appearing as the
point stabilizers.  We then give several equivalent descriptions of relative property A.  Our notion of relative property A
has the benefit that if $\mH$ consists of only the trivial subgroup, our definition reduces to
the statement that $G$ itself has property A.  From the work of Ozawa \cite{Oz}, groups acting cofinitely on
uniformly fine hyperbolic graphs provide examples of groups satisfying relative property A with respect to the vertex stabilizers.  
(In particular, this includes relatively hyperbolic groups, amalgamated free products, and HNN extensions.) 

In section 5 we give a cohomological characterization for relative property A.  In the case
that each of the subgroups are of property A, a long-exact sequence argument implies property A for the group.  We 
show that a finitely generated group acting cocompactly on a property A metric space has relative property A with 
respect to the point stabilizers.  It is known by \cite{C,BCGNW} that locally finite, finite dimensional CAT($0$) cube
complexes have property A, so that a group acting cocompactly on such a complex has relative property A with respect
to the vertex stabilizers.  Consequently, if each of the vertex stabilizers have property A then the group itself has property A.

Finally, in section 6 we propose a definition of relative amenability.  Our notion of relative amenability, in the case of a single 
subgroup, reduces to Monod and Popa's notion of co-amenability \cite{MoP}.  It is also equivalent to the amenability of
the quasi-regular representation of $G$ on $\ell^2(G/\mH)$ \cite{Be}.  A cohomological characterization for
relative amenability is obtained by specializing that for relative property A.  It is then shown that if a
group is relatively amenable with respect to a family of amenable subgroups, then the group is amenable.
We also show that a finitely generated group acting cocompactly on an amenable metric space with amenable
stabilizers is amenable.  

The second and third authors would like to thank the Department of Mathematical Sciences at IUPUI for their
support during our visit.  We would like to thank Nicolas Monod for bringing the papers \cite{Mo2, MoP} to our attention
and suggesting Proposition \ref{prop:higherDegrees}.  We would like to thank Thomas Sinclair for bring the paper \cite{Be}
to our attention.

\section{Bounded cohomology and Property A}

Suppose $X$ is a compact Hausdorff space, and denote by $C(X)$ the space of real-valued continuous functions
on $X$.   Let $K$ be a fixed countable set admitting a cofinite $G$ action with stabilizers the conjugates of 
elements of $\mH$.  Denote by $V$ be the space of all functions $f : K \to C(X)$ endowed with the norm
\[ \| f \|_{V} = \sup_{x \in X} \sum_{k \in K} |f_k(x)| \]
where $f_k \in C(X)$ is the function obtained by evaluating $f$ at $k \in K$.

Let $W_{00}(K,X)$ be the subspace of $V$ which contains all functions $f : K \to C(X)$ which have finite
support and such that for some $c \in \R$, $c = c(f)$, $\sum_{k \in K} f_k = c 1_X$, where
$1_X$ denotes the constant function with value $1$ on $X$.  Denote the closure of this subspace, in the $V$-norm,
by $W_{0}(K,X)$.  Let $\pi : W_{00}(K,X) \to \R$ be defined so that $\sum_{k \in K} f_k = \pi(f)1_X$.  By
continuity $\pi$ extends to all of $W_{0}(K,X)$. Denote the kernel of this extension by $N_{0}(K,X)$.

\begin{defn}
A Banach space $\E$ is a $C(X)$-module if it is equipped with a contractive unital representation
of $C(X)$.  If in addition $X$ admits a $G$-action, $\E$ is a $G$-$C(X)$ module if $G$ acts on
$\E$ by isometries and the representation of $C(X)$ is $G$-equivariant.
\end{defn}

\begin{defn}
Let $\E$ be a $G$-$C(X)$ module.  $\phi^1, \phi^2 \in \mathcal{E}$ are disjointly supported if there
exist $f_1, f_2 \in C(X)$ with disjoint support in $X$, with $f_1 \phi^1 = \phi^1$ and $f_2 \phi^2 = \phi^2$.

$\E$ is $\ell^\infty$-geometric if, whenever $\phi^1$ and $\phi^2$ are disjointly supported, 
$\| \phi^1 + \phi^2 \|_{\E} = \sup \left\{ \| \phi^1 \|_{\E}, \| \phi^2 \|_{\E} \right\}$.

$\E$ is $\ell^1$-geometric if, whenever $\phi^1$ and $\phi^2$ are disjointly supported, 
$\| \phi^1 + \phi^2 \|_{\E} = \| \phi^1 \|_{\E} + \| \phi^2 \|_{\E}$.
\end{defn}

\begin{lem}
The Banach space $V$ is a $G$-$C(X)$ module.  The subspace $N_0(K,X)$ is an $\ell^\infty$-geometric $G$-$C(X)$ module.
\end{lem}
\begin{proof}
The $G$ action on $X$ extends to an isometric action on $V$ via 
$(g \cdot f)_k(x) = f_{g^{-1}k}(g^{-1}x) = g \cdot (f_{g^{-1}k}(x))$.  Given a function $\phi \in C(X)$ and $f \in V$,
define $\phi \cdot f$ by $(\phi \cdot f)_k(x) = \phi(x) f_k(x)$.
This action restricts to $N_0(K,X)$.  It remains to show $N_0(K,X)$ is $\ell^\infty$-geometric.  

To that end, suppose $\phi^1$ and $\phi^2$ are disjointly supported elements of $N_0(K,X)$.
Then there exist disjointly supported $f_1$ and $f_2$ in $C(X)$ such that $f_1 \phi^1 = \phi^1$ and $f_2 \phi^2 = \phi^2$.
\begin{align*}
\| \phi^1 + \phi^2 \|_V &= \| f_1 \phi^1 + f_2 \phi^2 \| \\
	&= \sup_{x \in X} \sum_{k \in K} \left| f_1(x) \phi^1_k(x) + f_2(x) \phi^2_k(x) \right| \\
	&= \sup_{x \in X} \left(\sum_{k \in K} | f_1(x) \phi^1_k(x)| + \sum_{k \in K} | f_2(x) \phi^2_k(x)|  \right)\\
	&= \sup \left\{ \| \phi^1\|_V, \| \phi^2 \|_V  \right\} 
\end{align*}
since the two terms in the last sum are disjointly supported in $X$.
\end{proof}

The main result of \cite{BNNW} is the following theorem.
\begin{thm}[\cite{BNNW}(Theorem B)]\label{thm:BNNW}
Let $G$ be a countable discrete group acting by homeomorphisms on a compact Hausdorff space $X$.
Then the following are equivalent.
\begin{enumerate}
	\item The action of $G$ on $X$ is topologically amenable.
	\item The class $[J] \in H^1_b(G; N_0(G,X)^{**})$ is trivial.
	\item $H^p_b(G; \E^*) = 0$ for $p \geq 1$ and every $\ell^1$-geometric $G$-$C(X)$ module $\E$.
\end{enumerate}
\end{thm}
Here $[J]$ is a naturally constructed class in $H^1_b(G; N_0(G,X)^{**})$, and $N_0(G,X)$ is a
module serving as an absolute version of our relative coefficient module $N_0(K,X)$.

The main result of \cite{Mo2} has a similar characterization of topologically amenable actions
of groups.


\section{Relative bounded cohomology}
The cohomology of a group relative to a subgroup was defined by Auslander \cite{Au} for a single
not-necessarily normal subgroup, and by Bieri-Eckmann \cite{BE} for a family of subgroups.  The case of
relative bounded cohomology for a group relative to a family of subgroups was defined by Mineyev-Yaman
\cite{MY} and extended to more general bounding classes in \cite{JOR}.  We review the construction in this
section.

Suppose $\mH$ is a finite family of subgroups of the countable discrete group $G$.  For each 
$H \in \mH$, $G/H$ denotes the collection of left-cosets of $H$ in $G$. Let $G/\mH = \coprod_{H \in \mH} G/H$.
Then $\C G/\mH$ is the complex vector space with basis $G/\mH$.  Equivalently, $\C G/\mH$ is the complex
vector space of all finitely supported functions $G/\mH \to \C$.  Define the augmentation $\varepsilon : \C G/\mH \to \C$
by $\varepsilon( f ) = \sum_{x \in G/\mH} f(x)$, and set $\Delta = \ker \varepsilon$.
For $V$ a $\CG$ module, the relative cohomology of $G$ with respect to $\mH$ with coefficients in $V$ is given by
\[ H^p(G,\mH;V) \isom H^{p-1}(G; \Hom(\Delta, V)) \isom \Ext_{\CG}^{p-1}( \Delta, V ).  \]

The inclusion of coefficients $V \to \Hom( \C G/\mH , V ) \to \Hom(\Delta, V)$ induces a long-exact sequence
\[ \cdots \to H^p(G;V) \to H^p(\mH;V) \to H^{p+1}(G,\mH;V) \to H^{p+1}(G;V) \to \cdots  \]
where $H^p( \mH;V ) = \prod_{H \in \mH} H^p(H;V)$.

Relative bounded cohomology is constructed in a similar manner.  A bounded $\C G$ module is a
$\C G$ module which is a normed complex vector space such that the $G$ action is by uniformly
bounded linear operators.  The morphisms between these modules are the $\C G$ module morphisms
which are bounded with respect to the norms in their domain and range.  For two normed spaces $U$ and $V$, 
$\bHom(U,V)$ will denote the space of all bounded linear maps from $U$ to $V$.  If $U$ and $V$ are bounded
$\C G$ modules, $\bHom_{\C G}(U,V)$ will denote the bounded $\C G$ module maps from $U$ to $V$.  Denote
by $\bExt$ the $\Ext$ functor in this category of bounded $\C G$ modules and bounded $\C G$ morphisms.
Endow $\C G/\mH$ with the $\ell^1$-norm, $\| f \|_{\ell^1} = \sum_{x \in G/\mH} |f(x)|$, and $\Delta$
with the restriction of this norm.

The relative bounded cohomology of $G$ with respect to $\mH$ with coefficients in a bounded $\C G$ module $V$ is
given by \[ H^p_{b}(G,\mH;V) \isom H^{p-1}(G; \bHom(\Delta, V) ) \isom \bExt_{\CG}^{p-1}( \Delta, V ). \]

As shown in \cite{JOR} for more general bounding classes, there is a long-exact sequence in relative bounded
cohomology analogous to that above.
\[ \cdots \to H_b^p(G;V) \to H_b^p(\mH;V) \to H_b^{p+1}(G,\mH;V) \to H_b^{p+1}(G;V) \to \cdots  \]
where $H_b^p(\mH;V) = \prod_{H \in \mH} H_b^p(H;V)$.

When dealing with multiple subgroups, $\mH = \{ H_i \, | \, i \in \I \}$, we often work with a particular resolution.
Let $\I G = \coprod_{i \in \I} G_i$ the disjoint union of $G_i$, where each $G_i$ is a copy of $G$.
Let $S_k(\I G)$ be the collection of all $(k+1)$-tuples $(g_0, \ldots, g_k)$ where each $g_j \in \I G$.  Equip
$S_k( \I G)$ with the diagonal left $G$-action.  Set $St_k( \I G ) = \C S_k( \I G )$.  The sequence
\[ \cdots \to St_2( \I G) \to St_1( \I G ) \to St_0( \I G ) \to \C \to 0 \]
is exact, where the maps $\partial: St_k( \I G ) \to St_{k-1}(\I G)$ are the usual `skipping' boundary maps
\[ \partial( g_0, g_1, \ldots, g_k ) = \sum_{i = 0}^{k} (-1)^i (g_0, g_1, \ldots, \hat{g_i}, \ldots, g_k) \]
and $\epsilon: St_0(\I G) \to \C$ is the augmentation
\[ \epsilon(f) = \sum_{x \in \I G} f(x).\]  
Fix an $i \in \I$, and take $1$ to be the identity element of $G_i$.
Define $s : St_k( \I G ) \to St_{k+1}(\I G)$ by
$s( g_0, \ldots, g_k ) = (1, g_0, \ldots, g_k )$, and $s : \C \to St_0(\I G )$ by $s(z) = z(1)$.
This gives a contracting homotopy for the complex.
Thus $St_*( \I G )$ is a projective resolution of $\C$ over $\C G$.
This will be our standard resolution for calculating cohomologies.  When each $St_k( \I G )$ is endowed with
the $\ell^1$-norm, this also gives a bounded projective resolution of $\C$ over $\C G$, \cite{MY}, so it can
be utilized to calculate bounded cohomology as well.

\section{Relative property A}

Suppose $K$ is a countable cofinite $G$-set. Denote by $\mathcal{R}\subset K$ a set of representatives of the orbits of 
$G$ in $K$, and $\mH$ the set of subgroups of $G$ which occur as the stabilizers of the points in $\mathcal{R}$. 
(note that $\mH$ depends on the choice of $\mathcal{R}$).
\vspace{.2in}

For each $v \in K$, we may write $v$ as $v = g_v r_v$ for a unique $r_v\in \mathcal{R}$, with $g_v$ chosen so as to be of 
minimal length among all such elements translating $r_v$ to $v$. Define $\rho_{G,K}$ by the equality
\begin{align*}
&\rho_{G,K}: G\times K \to \R_+  \\
&\rho_{G,K}(g,v) = d_G(e,g_{g^{-1}v}) = \ell_G(g_{g^{-1}v})
\end{align*}
This map provides a way for measuring the distance between elements in $G$ and elements in $K$. In the particular case $H\normal G$ and $K = G/H$, $\rho_{G,K}(g,aH) = d_{G/H}(gH,aH) = d_{G/H}(\overline{g},\overline{a})$, so the above may be seen as a natural extension of the quotient metric to the case when one has an arbitrary cofinite $G$-set.


\begin{defn}\label{defn:YuRelPropA}
Suppose $G$ is a countable group and $\mH = \left\{ H_i \, | \, i \in \I \right\}$ is a finite family of subgroups of $G$.
$G$ has \underline{relative property A} with respect to $\mH$ if the following are satisfied.
\begin{enumerate}
	\item There exists a countable set $K$ admitting a cofinite $G$ action with point stabilizers precisely the conjugates of the elements of $\mH$.
	\item For every $\epsilon > 0$ and $R>0$ there exists an $S>0$ and a collection, $A_x$, of finite nonempty subsets of $K \times \N$ 
		indexed by $G$ such that the following hold.
		\begin{enumerate}
			\item For each $x \in G$, if $(k,j) \in A_x$ then $\rho_{G,K}(x,k) < S$. 
			\item For each $x, y\in G$ with $d_G(x,y) < R$ then \[ \frac{|A_x \Delta A_y|}{|A_x|} < \epsilon. \]
		\end{enumerate}
\end{enumerate}
\end{defn}

As is the case with property A, relative property A has myriad equivalent descriptions, some of which
we give below.  In each, the countable set $K$ is assumed to admit a cofinite $G$ action with stabilizers
the conjugates of the $H_i \in \mH$.  
The following is an analogue of a characterization of property A in \cite{BCGNW}.
\begin{prop}\label{prop:GBCNW_propADefn}
A discrete group $G$ has relative property A with respect to $\mH$ if and only if the following hold.
\begin{enumerate}
	\item There exists a countable space $K$ admitting a cofinite $G$ action with stabilizers the conjugates of elements of $\mH$.
	\item There exists a sequence of families of finitely supported functions $\xi_{n,x} : K \to \N \cup \{ 0 \}$, indexed by $x \in G$, satisfying the following conditions.
	\begin{enumerate}
		\item For every $n$ there is a constant $S_n$ such that if $\xi_{n,x}(k) \neq 0$, then $\rho_{G,K}(x,k) < S_n$.
		\item For every $R > 0$, \[ \frac{\| \xi_{n,y} - \xi_{n,x} \|_{\ell^1} }{\| \xi_{n,x} \|_{\ell^1}} \to 0\] 
			uniformly on the set $\{ (x,y) \, | \, d_G( x, y ) < R \}$ as $n \to \infty$.
	\end{enumerate}
\end{enumerate}
\end{prop}
\begin{proof}
It is clear that the existence of such a sequence is equivalent to the following.

For every $R>0$ and $\epsilon > 0$ there is an $S>0$ and a family of finitely supported functions
$\xi_x : K \to \N \cup \{ 0 \}$ satisfying these two conditions.
	\begin{enumerate}
		\item If $\xi_x(k) \neq 0$ then $\rho_{G,K}(x,k) < S$.
		\item If $d_G(x,y) < R$ then \[ \frac{\| \xi_{y} - \xi_{x} \|_{\ell^1} }{\| \xi_{x} \|_{\ell^1}} < \epsilon. \]
	\end{enumerate}

Assume $G$ has relative property A with respect to $\mH$, and fix $R > 0$ and $\epsilon > 0$.
Take $K$ and $A_x$ as in Definition \ref{defn:YuRelPropA}, and define $\xi_x : K \to \N \cup \{0\}$ by
\[ \xi_x(k) = \left| A_x \cap \left( {k} \times \N \right) \right|. \]
As such, $\| \xi_x \|_{\ell^1} = \left| A_x \right|$ and $\| \xi_y - \xi_x \|_{\ell^1} = \left| A_x \Delta A_y \right|$.
It is clear that the family $\xi_x$ satisfies the above conditions.

Conversely, fix $R>0$ and $\epsilon > 0$ and take $K$ and $\xi_x$ as above.
Set
\[ A_x = \left\{ (k,j) \, | \, 1 \leq j \leq \xi_x(k) \right\}. \]	
Then $\left| A_x \right| = \| \xi_x \|_{\ell^1}$ and $\left| A_x \Delta A_y \right| = \| \xi_y - \xi_x \|_{\ell^1}$.
Therefore $G$ has relative property A with respect to $\mH$.
\end{proof}

The following is a `Reiter's condition' type of characterization.
\begin{prop}\label{prop:ReitersConditionMany}
A discrete group $G$ has relative property A with respect to $\mH$ if and only if the following conditions are satisfied.
\begin{enumerate}
	\item There exists a countable space $K$ admitting a cofinite $G$ action with stabilizers the conjugates of elements of $\mH$.
	\item There exists a sequence of finitely supported functions $f_n : G \to \Prob(K)$ satisfying the following conditions.
	\begin{enumerate}
		\item For every $n$ there is a constant $S_n$ such that if $f_n(x)(k) \neq 0$, then $\rho_{G,K}(x,k) < S_n$.
		\item For every $R > 0$, $\| f_n(y) - f_n(x) \|_{\ell^1} \to 0$ as $n \to \infty$ uniformly on the
			set $\{ (x,y) \, | \, d_G( x, y ) < R \}$.
	\end{enumerate}
\end{enumerate}
\end{prop}
\begin{proof}
Suppose $G$ has relative property A with respect to $\mH$. Take $K$ and the sequence $\xi_{n,x} : K \to \N \cup \{0\}$ guaranteed
by Proposition \ref{prop:GBCNW_propADefn}.  Define $f_n : G \to \Prob(K)$ by $f_n(x) = \frac{1}{\| \xi_{n,x} \|_{\ell^1} } \xi_{n,x}$.
If $f_n(x)(k) \neq 0$ then $\xi_{n,x}(k) \neq 0$ so $\rho_{G,K}(x,k) < S_n$.
Moreover, 
\begin{align*}
\| f_n(x) - f_n(y) \|_{\ell^1} &= \| \frac{1}{\| \xi_{n,x} \|_{\ell^1} } \xi_{n,x} - 
		\frac{1}{\| \xi_{n,y} \|_{\ell^1} } \xi_{n,y} \|_{\ell^1}\\
	&=  \| \frac{1}{\| \xi_{n,x} \|_{\ell^1} } \xi_{n,x} - \frac{1}{\| \xi_{n,x} \|_{\ell^1} } \xi_{n,y} +
		 \frac{1}{\| \xi_{n,x} \|_{\ell^1} } \xi_{n,y} - \frac{1}{\| \xi_{n,y} \|_{\ell^1} } \xi_{n,y}\|_{\ell^1} \\
	&\leq \| \frac{1}{\| \xi_{n,x} \|_{\ell^1} } \xi_{n,x} - \frac{1}{\| \xi_{n,x} \|_{\ell^1} } \xi_{n,y} \|_{\ell^1} +
		 \| \frac{1}{\| \xi_{n,x} \|_{\ell^1} } \xi_{n,y} - \frac{1}{\| \xi_{n,y} \|_{\ell^1} } \xi_{n,y}\|_{\ell^1} \\
	&= \frac{ \| \xi_{n,x} - \xi_{n,y} \|_{\ell^1} }{ \| \xi_{n,x} \|_{\ell^1} } + 
		\left| \frac{1}{\| \xi_{n,x} \|_{\ell^1}} - \frac{1}{ \| \xi_{n,y} \|_{\ell^1}} \right| \| \xi_{n,y} \|_{\ell^1}\\
	&= \frac{ \| \xi_{n,x} - \xi_{n,y} \|_{\ell^1} }{ \| \xi_{n,x} \|_{\ell^1} } + 
		\left| \frac{\| \xi_{n,y}\|_{\ell^1} - \| \xi_{n,x} \|_{\ell^1}}{\| \xi_{n,x} \|_{\ell^1}} \right| \\
	&\leq \frac{ \| \xi_{n,x} - \xi_{n,y} \|_{\ell^1} }{ \| \xi_{n,x} \|_{\ell^1} } + 
		\frac{\| \xi_{n,y} - \xi_{n,x} \|_{\ell^1}}{\| \xi_{n,x} \|_{\ell^1}}\\
	&= 2 \frac{ \| \xi_{n,x} - \xi_{n,y} \|_{\ell^1} }{ \| \xi_{n,x} \|_{\ell^1} }	
\end{align*}		
This converges to $0$ uniformly as $n \to \infty$ on $\{ (x,y) \, | \, d_G(x,y) < R \}$ for all $R > 0$.

For the converse, assume such a $K$ and a sequence $f_n : G \to \Prob(K)$ exist.  Fix an $n > 0$.
As $f_n(x)(k) = 0$ whenever $\rho_{G,K}(x,k) \geq S_n$, each $f_n(x)$ is uniformly finitely supported.
By an approximation argument, we assume the existence of a positive integer $M$ such that for each $x$, $f_n(x)$ takes
only the values $0/M$, $1/M$, $\ldots$, $M/M$.
We define a function $\xi_{n,x} : K \to \N \cup \{ 0 \}$ by $\xi_{n,x}(k) = M f_n(x)(k)$
for $k \in K$.  By the properties of $f_{n,x}$ if $\rho_{G,K}(x,k) \geq S_n$ then $\xi_{n,x}(k) = 0$.
Also $\| \xi_{n,y} - \xi_{n,x} \|_{\ell^1}  = M \| f_n(y) - f_n(x) \|_{\ell^1}$ with $\| \xi_{n,x} \|_{\ell^1} = M$.
Thus for all $R > 0$,
\[ \frac{\| \xi_{n,y} - \xi_{n,x} \|_{\ell^1} }{\| \xi_{n,x} \|_{\ell^1}} = \| f_{n}(y) - f_{n}(x) \|_{\ell^1} \to 0\] 
uniformly on the set $\{ (x,y) \, | \, d_G( x, y ) < R \}$ as $n \to \infty$.
\end{proof}

In regards to the `amenable action on a compact Hausdorff space' characterization of property A, relative property A can also be
characterized in terms of its action on a compact Hausdorff space.
\begin{prop}\label{prop:OzawaRelPropAMany}
A countable group $G$ has relative property A with respect to $\mH$, if and only if the following hold.
\begin{enumerate}
	\item There exists a countable space $K$ admitting a cofinite $G$ action with stabilizers the conjugates of elements of $\mH$.
	\item There exist a compact Hausdorff $G$-space, $X$, and a sequence of weak$^*$-continuous functions, 
		$\xi_n : X \to \Prob(K)$, such that for all $g \in G$,  
		\[\lim_{n\to\infty} \sup_{x \in X} \| g \xi_n(x) - \xi_n( gx ) \|_{\ell^1} = 0.\]  
\end{enumerate}
\end{prop}
\begin{proof}
We use Proposition \ref{prop:ReitersConditionMany} and argue as in the proof of Lemma 3.3 of \cite{HR}.
As in \cite{HR}, we may assume $X$ is $\BG$, the Stone-Cech compactification of $G$.
Let $K$ and $\xi_n : \BG \to \Prob(K)$ be as in the statement of the lemma.  For each $n$, 
define $b_n : G \to \Prob(K)$ as $b_n(g) = \xi_n(g)$, for $g \in G \subset \BG$.  As the image of $b_n$ sits inside a
weak$^*$-compact subset of $\Prob(K)$, by Lemma 3.8 of \cite{HR} we may assume that there is a finite
$J \subset K$ on which each $b_n(g)$ is supported.  Define $f_n : G \to \Prob(K)$ by
$f_n(g)= g b_n(g^{-1})$.  If $d_G(x,y) < R$, then we may write $y = x g$ for $d_G( e, g ) < R$.
\begin{align*}
	\| f_n(y) - f_n(x) \|_{\ell^1} &= \| y b_n( y^{-1} ) - x b_n( x^{-1} ) \|_{\ell^1}\\
		&= \| y \xi_n( y^{-1} ) - x \xi_n( x^{-1} ) \|_{\ell^1}\\
		&= \| x g \xi_n( g^{-1} x^{-1} ) - x \xi_n( x^{-1} ) \|_{\ell^1}\\
		&= \| g \xi_n( g^{-1} x^{-1} ) - \xi_n( x^{-1} ) \|_{\ell^1}\\
		&= \|  \xi_n( g^{-1} x^{-1} ) - g^{-1} \xi_n( x^{-1} ) \|_{\ell^1}.
\end{align*}
As $d_G$ is proper on $G$, we have that $\| f_n(y) - f_n(x) \|_{\ell^1} \to 0$ uniformly
on $\{ (x,y) \, | \, d_G( x, y ) < R \}$ as $n \to \infty$.
Moreover, $f_n(x)$ is supported only on $xF \subset K$. For $v \in xF$, $x^{-1} v \in F$ so
$d_G( e, g_{x^{-1}v} ) < L$ for some constant $L$, as $F$ is finite.
Set $S_n = L$.  Then if $\rho_{G,K}(x,v) \geq S_n$, $v \notin xF$ so $f_n(x)(v) = 0$.
By Proposition \ref{prop:ReitersConditionMany}, $G$ has relative property A with respect to $\mH$.

For the converse, assume $G$ has relative property A with respect to $\mH$.
Let $K$ and $f_n : G \to \Prob(K)$ be given by Proposition \ref{prop:ReitersConditionMany}.
Define $a_n : G \to \Prob(K)$ by $a_n(g) = g f_n(g^{-1})$.  Thus 
\begin{align*}
	\| g a_n(x) - a_n(gx) \|_{\ell^1} &= \| g x f_n(x^{-1}) - g x f_n( x^{-1} g^{-1} ) \|_{\ell^1}\\
		&= \| f_n( x^{-1} ) - f_n( x^{-1} g^{-1} ) \|_{\ell^1}.
\end{align*}
As $d_G( x^{-1}, x^{-1} g^{-1} ) = d_G( e, g^{-1} )$ for all $x \in G$, we have
\[ \lim_{n \to \infty} \sup_{x \in G} \| g a_n(x) - a_n(gx) \|_{\ell^1} = 0. \]
Further, if $\rho_{G,K}(e,k) > S_n$, $f_n(x^{-1})(x^{-1}k) = 0$.  This implies that for every $n$
there is a finite subset $F' \subset K$ such that $a_n(x)$ is supported in $F'$ for all
$x$.  As the image of $a_n$ sit in a weak$^*$-compact subset of $\Prob(K)$, these maps extend 
to $\xi_n : \BG \to \Prob(K)$ by the universality property of $\BG$.  That the $\xi_n$ satisfy
the properties above follow from the density of $G$ in $\BG$.
\end{proof}

We note that by the proof of Proposition 11 of \cite{Oz}, the assumption of `weak$^*$-continuous'
here can be relaxed to `Borel'.  

Recall that a finitely generated group is relatively hyperbolic with respect to a family of subgroups if the
coned-off graph is hyperbolic and satisfies the bounded coset penetration property \cite{Farb}.  By \cite{Oz},
relatively hyperbolic groups have relative property A with respect to their peripheral subgroups.
The compact space $X$ can be taken to be $\Delta \hat \Gamma$, the Gromov compactification of the 
associated coned-off Cayley graph  $\hat \Gamma$, and $K$ can be taken to be the vertex set of $\hat \Gamma$.

The following is a variant on the definition of a relatively amenable action, following \cite{BNNW}.
\begin{defn}\label{def:BNNWRelAmenabAction}
Suppose $X$ is a compact Hausdorff space admitting a $G$ action.  The action is amenable relative to the
subgroups $\mH$ if the following conditions are satisfied.
\begin{enumerate}
	\item There exists a countable set $K$ admitting a cofinite $G$ action with stabilizers the conjugates of elements of $\mH$.
	\item There exists a sequence of elements $f^n \in W_{00}(K,X)$ such that
		\begin{enumerate}
			\item $f^n_k \geq 0$ for all $n \in \N$ and $k \in K$.
			\item $\pi(f^n) = 1$ for every $n \in \N$.
			\item for each $g \in G$ we have $\| f^n - g f^n \|_{V} \to 0$.
		\end{enumerate}
\end{enumerate}
\end{defn}

\begin{prop}
Suppose $G$ acts on a compact Hausdorff space $X$.  Then $G$ satisfies condition (2) of Proposition \ref{prop:OzawaRelPropAMany} 
with respect to  $\mH$ if and only if this action is relatively amenable with respect to $\mH$.  In particular admitting
a relatively amenable action on a compact Hausdorff space is equivalent to relative property A.
\end{prop}

\begin{proof}
We first assume that $G$ has relative property A with respect to $\mH$.  This gives the sequence $\xi_n : X \to \Prob(K)$
of weak$^*$ continuous functions as above.  Define $S_n : K \to C(X)$ by $S_n(k)(x) = \xi_n(x)(k)$.  Then for each $x \in X$,
\[ \sum_{k \in K} S_n(k)(x) = \sum_{k \in K} \xi_n(x)(k) = 1. \]
Moreover for each $g \in G$,
\begin{align*}
 \| S_n - g S_n \|_V &= \sup_{x \in X} \sum_{k \in K} | S_n(k)(x) - (g S_n)(k)(x)|\\
 &= \sup_{x \in X} \sum_{k \in K} |\xi_n(x)(k) - (g \xi_n)(x)(k) |\\
 &= \sup_{x \in X} \sum_{k \in K} |\xi_n(x)(k) - (g \xi_n(g^{-1}x))(k) |\\
 &= \sup_{y \in X} \sum_{k \in K} | \xi_n(gy)(k) - (g \xi_n(y) )(k) |\\
 &= \sup_{y \in X} \|  \xi_n(gy) - g \xi_n(y) \|_{\ell^1} 
\end{align*}
Thus, this tends to zero as $n \to \infty$.  The $S_n$ so constructed need not be finitely supported, 
however they can be approximated by finitely supported functions.  Let $\mathcal{F}_{j}$ denote an 
increasing sequence of finite subsets of $K$ with $\cup \mathcal{F}_j = K$.  As for each $n$, 
$\sum_{k \in K} S_n(k) = 1_X$, for all sufficiently large $j$, $\sum_{k \in \mathcal{F}_j} S_n(k) > 0$.  
In particular the sum is bounded away from zero. For each $n$, let $j_n$ be one such sufficiently large value of $j$.
For each $j$, set \[ S_{n,j}(k) = \frac{1}{\sum_{k \in \mathcal{F}_j} S_n(k)} S_n(k) \]
for $k \in \mathcal{F}_j$, and $S_{n,j}(k) = 0$ for $k \notin \mathcal{F}_j$.  
Setting $f^n = S_{n,j_n}$ we find the action is relatively amenable.

For the other direction, given the sequence $f^n$ from the definition of a relatively amenable action,
define $\xi_n : X \to \Prob(K)$ by $\xi_n(x)(k) = f^n_k(x)$.  Reversing the above process yields
the result.
\end{proof}

\begin{defn}
Let $G$ be a countable group acting on a compact Hausdorff space $X$ by homeomorphisms.  A relative mean for the action,
with respect to the finite family of subgroups $\mH$, is an element $\mu \in W_{0}(K,X)^{**}$ such that
$\mu( \pi ) = 1$.   A relative mean $\mu$ is invariant if $\mu( g \phi ) = \mu( \phi )$ for every
$\phi \in W_{0}(K,X)^{*}$.
\end{defn}

\begin{lem}\label{lem:InvMeanRelPropA}
Let $G$ be a countable group acting on a compact Hausdorff space $X$ by homeomorphisms.  The action is relatively amenable
with respect to $\mH$ if and only if there is an invariant relative mean for the action with respect to $\mH$.
\end{lem}
\begin{proof}
This is almost verbatim from \cite{BNNW}. We include the proof for convenience. 

Suppose $G$ acts amenably on $X$ relative to $\mH$, and consider the sequence $f^n \in W_0(K,X)$ from Definition
\ref{def:BNNWRelAmenabAction}.  View the $f^n$ as elements of $W_0(K,X)^{**}$.  The unit ball in $W_0(K,X)^{**}$
is weak$^*$ compact, so there is a convergent subsequence $f_{n_k}$.  Let $\mu$ be the weak$^*$ limit of this
subsequence.  $\mu( \pi ) = 1$ since for each $n$ $\langle f^n , \pi \rangle = 1$.  Also 
\[| \langle f^{n_k} - g f^{n_k}, \phi \rangle | \leq \| f^{n_k} - g f^{n_k} \|_V \| \phi \|.\]  As the right-hand side
tends to zero, $\mu( g \phi ) = \mu( \phi )$.

For the converse, by Goldstine's theorem $\mu \in W_0(K,X)^{**}$ is the weak$^*$ limit of a bounded net of elements
$f^\lambda \in W_0(K,X)$.  Moreover, we can assume $\pi( f^\lambda ) = 1$.
As $\mu$ is invariant, $f^\lambda - g f^\lambda \to 0$ in the weak$^*$ topology.  As $f^\lambda - g f^\lambda$
are actually in $W_0(K,X)$ this convergence is in the weak topology on $W_0(K,X)$.

For each $\lambda$, consider $(f^\lambda - g f^\lambda )_{g \in G}$ as an element of $\prod_{g \in G} W_0(K,X)$.
In this space, this sequence of elements tends to zero in the product weak topology.  $\prod_{g \in G} W_0(K,X)$ 
is a Fr\'echet space in the product norm topology, so by Mazur's theorem there is a sequence $f^n$ of convex 
combinations of the $f^\lambda$ such that $(f^n - g f^n)_{g \in G}$ converges to zero in the Fr\'echet topology.  
Thus there exists a sequence $f^n$ of elements of $W_0(K,X)$ such that for every $g \in G$, $\| f^n - g f^n \| \to 0$.
\end{proof}

\begin{lem}
Suppose $G$ is a countable group with relative property A with respect to a finite family of subgroups, $\mH$.  
There exist a compact Hausdorff $G$-space, $X$, and a sequence of weak$^*$-continuous functions, 
		$\zeta_n : X \to \Prob(G/\mH)$, such that for all $g \in G$,  
		\[\lim_{n\to\infty} \sup_{x \in X} \| g \zeta_n(x) - \zeta_n( gx ) \|_{\ell^1} = 0.\]
In particular, $K$ in the definition of relative property A can be taken to be $G/\mH$.
\end{lem}
\begin{proof}
The $G$ action on $G/\mH$ is cofinite with point stabilizers the conjugates of the elements of $\mH$.
Take a sequence of weak$^*$ continuous maps $\xi_n : \BG \to \Prob(K)$ as in Proposition \ref{prop:OzawaRelPropAMany}.
Construct a map $\pi : K \to G/\mH$ as follows.  For $u \in U$ with stabilizer $H_i$, set $\pi(u) = H_i$.
For other $k \in K$, there is a $u \in U$ and $g \in G$ with $k = gu$.  Set $\pi(k) = g \pi(u)$.
This is a well-defined, surjective $G$-map.

Define $\zeta_n : X \to \Prob(G/\mH)$ by
\[ \zeta_n(x)(y) = \sum_{\stackrel{k \in K}{\pi(k) = y}} \xi_n(k). \]
For $g \in G$ and $x \in X$,
\begin{align*}
(g \zeta_n(x))(y) - \zeta_n(gx)(y) &= \zeta_n(x)(g^{-1}y) - \zeta_n(gx)(y)\\
	&= \sum_{\stackrel{k \in K}{\pi(k) = g^{-1}y}} \xi_n(x)(k) - \sum_{\stackrel{k \in K}{\pi(k) = y}} \xi_n(gx)(k)\\
	&= \sum_{\stackrel{k \in K}{\pi(k) = y}} \xi_n(x)(g^{-1}k) - \sum_{\stackrel{k \in K}{\pi(k) = y}} \xi_n(gx)(k)\\
	&= \sum_{\stackrel{k \in K}{\pi(k) = y}} \left(  (g\xi_n(x))(k) - \xi_n(gx)(k) \right). \\
\| g \zeta_n(x) - \zeta_n(gx) \|_{\ell^1} &= \sum_{y \in G/\mH} |(g \zeta_n(x))(y) - \zeta_n(gx)(y)| \\
	&\leq \sum_{y \in G/\mH} \left| \sum_{\stackrel{k \in K}{\pi(k) = y}}\left( (g\xi_n(x))(k) - \xi_n(gx)(k) \right)\right| \\
	&\leq \sum_{y \in G/\mH} \sum_{\stackrel{k \in K}{\pi(k) = y}} | ( g\xi_n(x))(k) - \xi_n(gx)(k) | \\
	&= 	\sum_{k \in K} | ( g\xi_n(x))(k) - \xi_n(gx)(k) | \\
	&= \| g\xi_n(x) - \xi_n(gx) \|_{\ell^1}.
\end{align*}

Thus $\lim_{n \to \infty} \sup_{x \in X} \| g \zeta_n(x) - \zeta_n( gx ) \|_{\ell^1} = 0$.
\end{proof}
Due to this lemma, we will always be assuming $K = G/\mH$ unless stated otherwise.

\begin{prop}\label{prop:Extensions}
Suppose $H \normal G$.  If the quotient group $G/H$ has property A, then $G$ has relative property A with respect to $H$.
\end{prop}
\begin{proof}
We endow $H$ with the restricted length from $G$, and $Q = G/H$ with the quotient length,
$\ell_Q(qH ) = \min \{ \ell_G( qh ) \, | \, h \in H \}$.

Suppose $Q$ has property A.  Then $Q$ acts topologically amenably on its Stone-Cech compactification $\beta Q$.
Let $\xi_n : \beta Q \to \Prob(Q)$ be a sequence of weak$^*$ continuous functions such that for all $qH \in Q$
\[ \lim_{n \to \infty} \sup_{x \in \beta Q} \| qH \xi_n(x) - \xi_n( qH x ) \|_{\ell^1} = 0. \]
The $G$ action on $Q$ by isometries extends to a $G$ action on $\beta Q$.  For any $g \in G$ and
any $qH \in Q$, $g qH = gH qH$, with $gH \in Q$.  In particular, for any $x \in \beta Q$
$\| g \xi_n(x) - \xi_n( gx ) \|_{\ell^1} = \| gH \xi_n(x) - \xi_n( gH x ) \|_{\ell^1}$.
Then for all $g \in G$,
\[ \lim_{n \to \infty} \sup_{x \in \beta Q} \| g \xi_n(x) - \xi_n( g x ) \|_{\ell^1} = 0. \]
Therefore $G$ has relative property A with respect to $H$.
\end{proof}

\begin{rmk*}
The converse of this proposition is not true.  Consider the example of $Q$ a finitely generated group without property A, \cite{Oz2},
and $G$ a finite rank free group projecting onto $Q$, with kernel $H$.  Corollary \ref{cor:PropAGivesRelPropA} below
shows that if $G$ has property A then
$G$ has relative property A with respect to $H$, which contradicts the converse.  This shows that the notion of relative
property A is fundamentally different than the quotient having property A.
\end{rmk*}

\begin{prop}
Suppose $\mH$ is a finite family of finite index subgroups of the countable group $G$.  Then $G$ has relative property $A$ with
respect to $\mH$.
\end{prop}
\begin{proof}
Let $K = G/\mH$, and let $p$ be the uniform probability measure on the finite set $K$, and
let $\xi_n : \beta G \to \Prob(K)$ be the sequence of maps defined by $\xi_n(x) = p$
for all $n$ and $x$.  It is obvious that $\xi_n$ is weak$^*$ continuous and
\[ \lim_{n \to \infty} \sup_{x \in \beta G} \| g \xi_n(x) - \xi_n( gx ) \|_{\ell^1} = 0 \]
for all $g \in G$.
\end{proof}

\section{A cohomological characterization}
For each $i \in \I$, fix a vertex $k_i$ in $K$ such that the stabilizer
of $k_i$ in $G$ is $H_i$.  Then for $g \in \I G$, set $k_g = k_i$ for
$g \in G_i$.  Note that when $K = G/\mH$, $k_i$ is the coset $H_i$, and if $g \in G_i$
then $k_g$ is the coset $H_i$.

\begin{defn}
Let $G$ be a countable discrete group and let $H$ be a subgroup of $G$.  Assume that
$G$ acts by homeomorphisms on a compact Hausdorff topological space $X$.  Define
\[ J_r : St_1(\I G) \to N_{00}(K,X) \]
by $J_r( g_0, g_1 ) = \delta_{g_1 k_{g_1}} - \delta_{g_0 k_{g_0}}$.
\end{defn}
The function $J_r$ so defined is a bounded cocycle, thus represents a class $[J_r]$ in
$H^1_b(G; N_0(K,X)^{**} )$.  In the case that $\mH$ consists of a single subgroup which
is trivial, $[J_r]$ reduces to the Johnson class $[J]$ of \cite{BNNW}.  The class $[J_r]$
is called the relative Johnson class.

\begin{thm}\label{thm:RelJohnsonClass}
The class $[J_r] \in H^1_b(G; N_0(K,X)^{**})$ is trivial if and only if the action of
$G$ on $X$ is relatively amenable with respect to $\mH$.
\end{thm}
\begin{proof}
The short exact sequence \[ 0 \to N_0(K,X) \to W_0(K,X) \stackrel{\pi}{\to} \C \to 0 \]
gives the short exact sequence
\[ 0 \to N_0(K,X)^{**} \to W_0(K,X)^{**} \to \C \to 0. \]
This yields the following long-exact sequence.
\[ \cdots \to H^0_b(G;\C) \to H^1_b(G; N_0(K,X)^{**}) \to H^1_b(G; W_0(K,X)^{**} ) \to \cdots \]
The class $[J_r]$ is the image under the connecting homomorphism 
$\delta: H^0_b(G;\C) \to H^1_b(G; N_0(K,X)^{**})$ of the class $[1] \in H^0_b(G;\C)$ represented
by the constant function with value $1$ on $G$, $[J_r] = \delta [1]$.
By the exactness of the long exact sequence, $[J_r] = 0$ if and only if $[1] \in \im \pi^{**}$
where $\pi^{**} : H^0_b(G; W_0(K,X)^{**} ) \to H^0_b(G;\C)$ is induced by 
$\pi : W_0(K,X) \to \C$ as above.  As $H^0_b(G; W_0(K,X)^{**}) = ( W_0(K,X)^{**} )^G$ and
$H^0_b(G; \C) = \C$, $[J_r] = 0$ if and only if there exists an element $\mu \in W_0(K,X)^{**}$
with $\mu = g \mu$ and $\mu( \pi ) = 1$.
The equivalence now follows from Lemma \ref{lem:InvMeanRelPropA}.
\end{proof}

\begin{prop}\label{prop:RelJohnsonInRelCoh}
The image of $[J_r]$ under the restriction $H^1_b(G; N_0(K,X)^{**}) \to H^1_b(\mH; N_0(K,X)^{**})$
is trivial.  In particular $[J_r]$ lies in the image of the map
$H^1_b(G, \mH; N_0(K,X)^{**}) \to H^1_b(G; N_0(K,X)^{**})$.
\end{prop}
\begin{proof}
Suppose $H \in \mH$. 
For $h_0, h_1 \in H$, $J_r(h_0, h_1) = \delta_{h_1 k_0} - \delta_{h_0 k_0}$.
As the stabilizer of $k_0$ in $G$ is $H$, this difference is $0$.
\end{proof}

Actually, more is shown in this proposition.  The restriction of $J_r$ to the subgroups is
identically zero.  Thus $J_r$ is a relative cocycle, not merely an absolute cocycle.

\begin{lem}
Let $G$ be a countable group, $\mH$ a family of subgroups of $G$.  Assume that
$G$ acts by homeomorphisms on a compact Hausdorff topological space $X$.
If the map $H^0_b(\mH; N_0(K,X)^{**}) \to H^1_b(G,\mH; N_0(K,X)^{**} )$ is surjective then the action of $G$ on $X$ is
relatively amenable with respect to $\mH$.
\end{lem}
\begin{proof}
By Proposition \ref{prop:RelJohnsonInRelCoh}, $[J_r]$ is in the image of the map
$H^1_b(G, \mH; N_0(K,X)^{**}) \to H^1_b(G; N_0(K,X)^{**})$.  If $H^0_b(\mH; N_0(K,X)^{**}) \to H^1_b(G,\mH; N_0(K,X)^{**} )$
is surjective then $[J_r] = 0$.  The action is relatively amenable by Theorem \ref{thm:RelJohnsonClass}.
\end{proof}

Assume $\E$ is an $\ell^1$-geometric $G$-$C(X)$ module, and let $\tau \in \ell^\infty(K, \E^*)$.
Pick a vector $v \in \E$ and define a linear functional $\sigma_{\tau, v} : W_{00}(K,X) \to \C$ by
\[ \sigma_{\tau, v}(f) = \langle \sum_{k \in K} f_k \tau_k, v \rangle \]

\begin{defn}
Let $\E$ be an $\ell^1$-geometric $G$-$C(X)$ module, and let $\mu \in W_0(K,X)^{**}$ be a relative
invariant mean for the action.  Define a map $\mu_\E : \ell^\infty( K, \E^*) \to \E^*$ by
\[ \langle \mu_{\E}(\tau), v \rangle = \langle \mu, \sigma_{\tau, v} \rangle \]
for every $\tau \in \ell^\infty( K, \E^*)$ and $v \in \E$.
\end{defn}

\begin{lem}
Let $\E$ be an $\ell^1$-geometric $G$-$C(X)$ module, and let $\mu \in W_0(K,X)^{**}$ be an
invariant mean for the action.
\begin{enumerate}
	\item $\mu_\E$ is $G$-equivariant.
	\item If $\tau \in \ell^\infty(K, \E^* )$ is constant, then $\mu_{\E}( \tau ) = \tau_{H}$.  (As $\tau$ is
			constant, any $H \in \mH$ will give the same result. )
\end{enumerate}
\end{lem}
\begin{proof}
This follows immediately as in \cite{BNNW}.
\end{proof}

\begin{thm}
Let $G$ be a countable group, $\mH$ a family of subgroups of $G$.  Assume that
$G$ acts by homeomorphisms on a compact Hausdorff topological space $X$, that the action of $G$ on $X$ is relatively 
amenable with respect to $\mH$, and that $\E$ is an $\ell^1$-geometric $G$-$C(X)$ module.  Then the map 
$H^0_b(\mH; \E^{*}) \to H^1_b(G,\mH; \E^{*} )$ is surjective.
\end{thm}
\begin{proof}
$H^0_b(\mH; \E^*) \isom \bHom_{\CG}( \C G/\mH, \E^* )$, $H^1_b(G, \mH; \E^* ) \isom \bHom_{\CG}( \Delta, \E^* )$,
and the natural map $H^0_b(\mH; \E^*) \to H^1_b(G, \mH; \E^*)$ is induced by the restriction
$\bHom( \C G/\mH, \E^* ) \to \bHom( \Delta, \E^* )$.

For $\phi \in \bHom_{\CG}(\Delta, \E^* )$, define $\hat \phi : G/\mH \to \ell^\infty(G/\mH, \E^*)$ by
$(\hat \phi(k))(k') = \phi( k - k' )$.

For $g \in G$ and $k, k' \in G/\mH$, 
\begin{align*}
(\hat \phi(gk))(k') &= \phi( gk - k' )\\
	&= g \cdot \phi( k - g^{-1}k' )\\
	&= g \cdot ( \hat \phi(k) )(g^{-1}k' )\\
	&= ( g \hat \phi(k) )(k')
\end{align*}
In particular $\hat \phi(gk) = g \hat \phi(k)$ so $\hat \phi$ is $G$ equivariant.

Define a map $s : \bHom_{\CG}(\Delta, \E^*) \to \bHom_{\CG}(\C G/\mH, \E^*)$ via
\[ (s\phi)(k) = \mu_\E( \hat \phi(k) ). \]
Since $\hat \phi$ and $\mu_\E$ are $G$-equivariant we do have $s\phi \in \bHom_{\CG}(\C G/\mH, \E^* )$.

We consider $s\phi$ when restricted to $\Delta \subset \C G/\mH$.
\begin{align*}
(s\phi)(k-k') &= (s\phi)(k) - (s\phi)(k')\\
	&= \mu_\E( \hat\phi(k) ) - \mu_\E(\hat\phi(k'))\\
	&= \mu_\E( \hat\phi(k) - \hat\phi(k') )
\end{align*}
Here  \[(\hat\phi(k)-\hat\phi(k'))(w) = \phi(k - w) - \phi(k' - w ) = \phi( k - k' ). \]
In particular, $(\hat\phi(k) - \hat\phi(k'))(w)$ is independent of $w$.  Thus
$(s\phi)(k - k' ) = \phi(k -k' )$.  

This shows for every $\phi \in \bHom_{\CG}(\Delta, \E^*)$, there is an $s\phi \in \bHom_{\CG}(\C G/\mH, \E^*)$
with $s\phi$ restricting to $\phi$.  In particular, the map $H^0_b(\mH; \E^*) \to H^1_b(G, \mH; \E^*)$
is surjective.
\end{proof}

As $N_0(K,X)^*$ is an $\ell^1$-geometric $G$-$C(X)$ module, we have established the following.

\begin{cor}\label{cor:cohomchar}
Let $G$ be a countable group, $\mH$ a family of subgroups of $G$.  
The following are equivalent.
\begin{enumerate}
	\item The action of $G$ has relative property A with respect to $\mH$.
	\item For every $\ell^1$-geometric $G$-$C(X)$ module $\E$, the map $H^0_b(\mH;\E^*) \to H^1_b(G,\mH;\E^*)$ is surjective.
\end{enumerate}
\end{cor}

\begin{cor}\label{cor:RelativePropAGivesPropA}
Let $G$ be a countable group, $\mH$ a family of subgroups of $G$.  Suppose
that $G$ has relative property A with respect to $\mH$.
Then $G$ has property A if and only if each subgroup in $\mH$ has property A.
\end{cor}
\begin{proof}
If $G$ has property A, it is well known that each subgroup of $G$ also has property A.
For the converse, suppose each $H_i \in \mH$ has property A.  For each $\ell^1$-geometric $G$-$C(X)$ module
$\E$ consider the long-exact sequence
\[ \cdots \to H^0_b(\mH; \E^*) \stackrel{\delta}{\to} H^1_b(G,\mH; \E^* ) 
	\stackrel{d}{\to} H^1_b(G; \E^*) \stackrel{r}{\to} H^1_b( \mH; \E^* ) \to \cdots. \]
As $G$ has relative property A with respect to $\mH$, $\delta$ is surjective, thus $d$ is the zero map and $r$ is injective.
Since each $H_i \in \mH$ has property A, $H^1_b(\mH; \E^*) = 0$, thus $H^1_b(G;\E^*)$ is trivial.  
The result follows from Theorem \ref{thm:BNNW}.
\end{proof}

\begin{cor}\label{cor:PropAGivesRelPropA}
Let $G$ be a countable group.  If $G$ has property A, then $G$ has relative
property A with respect to any finite family of subgroups.
\end{cor}
\begin{proof}
Suppose $\E$ is an $\ell^1$-geometric $G$-$C(X)$ module, and consider the long-exact sequence
\[ \cdots \to H^0_b(\mH; \E^*) \stackrel{\delta}{\to} H^1_b(G,\mH; \E^* ) 
	\stackrel{d}{\to} H^1_b(G; \E^*) \stackrel{r}{\to} H^1_b( \mH; \E^* ) \to \cdots. \]
If $G$ has property A, then $H^1_b(G; \E^* ) = 0$, whence $H^0_b(\mH; \E^*) \to H^1_b(G,\mH;\E^*)$
is surjective.  As this holds for all $\E$, the result follows from Corollary \ref{cor:cohomchar}.
\end{proof}

From Corollary \ref{cor:RelativePropAGivesPropA} and Proposition \ref{prop:Extensions} we immediately
obtain the following well-known theorem.
\begin{cor}
Suppose $G$ is an extension of the countable group $H$ by the countable group $K$.  Then $G$ has 
property A if and only if each $H$ and $K$ have property A.
\end{cor}

\begin{thm}\label{thm:PropAComplexRelPropA}
Suppose the finitely generated group $G$ acts cocompactly on a uniformly discrete metric space
with bounded geometry, $(X, d_X)$.  Pick a family of representatives $\mathcal{R}$ of orbits of $G$ in
$X$.  Let $\mH$ be the family of subgroups which each stabilize an element of $\mathcal{R}$.
If $(X,d_X)$ has property A, then $G$ has relative property A with respect to $\mH$.
\end{thm}
\begin{proof}
If $X$ has property A, then there exist a sequence of functions $\xi_n : X \to \Prob(X)$
and a sequence of constants $S_n > 0$ such that the following are satisfied.
\begin{enumerate}
	\item For each $n$ and $x$, the support of $\xi_n(x)$ is contained in the ball $B_{S_n}(x)$.
	\item For each $R>0$, $\| \xi_n(x) - \xi_n(y) \|_{\ell^1} \to 0$ uniformly on the set
		$\{ (x,y) \, | \, d_X(x,y) < R \}$.
\end{enumerate}

Fix a basepoint $x_0 \in X$, and consider the functions $\mu_n : G \to \Prob(X)$ defined by
\[ \mu_n(g) := g \xi_n(g^{-1}x_0 ). \]

\begin{align*}
	\| g \mu_n(w) - \mu_n(gw) \|_{\ell^1} &= \| g w \xi_n( w^{-1} x_0 ) - gw \xi_n( w^{-1} g^{-1} x_0 ) \|_{\ell^1} \\
		&= \| \xi_n( w^{-1} x_0 ) - \xi_n( w^{-1} g^{-1} x_0 ) \|_{\ell^1}
\end{align*}

There is a constant $C = C(G,X)$ such that $d_X( x_0, gx_0 ) \leq C \ell_G( g )$ for all $g \in G$.  Thus for a fixed
$g \in G$, $d_X( w^{-1} x_0, w^{-1} g^{-1} x_0 ) = d_X( g x_0, x_0 ) \leq C \ell_G(g)$ is bounded independent of $w \in G$.
By the second condition above we have
\[ \lim_{n \to \infty} \sup_{w \in G} \| g \mu_n(w) - \mu_n( gw ) \| = 0. \]

Examine $k \in X$ which lies in the support of $\mu_n(g)$.
$\mu_n(g)(k) \neq 0$ if and only if $g \xi_n (g^{-1} x_0 ) (k) = \xi_n( g^{-1} x_0 ) (g^{-1}k) \neq 0$.
This implies $d_X( g^{-1} x_0, g^{-1} k ) = d_X(x_0, k) < S_n$.  By uniform boundedness we have a finite set $F \subset X$
such that each $\mu_n(g)$ is supported in $F$.  The collection of all probability measures supported on $F$ is weak$^*$-compact
so $\mu_n$ extends to a weak$^*$ continuous function $\mu_n : \beta G \to \Prob(X)$ which satisfy
\[ \lim_{n \to \infty} \sup_{x \in \beta G} \| g \mu_n(x) - \mu_n( gx ) \| = 0. \]
\end{proof}

This gives an extension of a theorem of Bell, \cite{B}, and is similar in spirit to a result of Guentner-Tessera-Yu, \cite[Corollary 3.2.4]{GTY}.
\begin{cor}\label{cor:ActionPropA}
Suppose the finitely generated group $G$ acts cocompactly on uniformly discrete, bounded geometry metric space $X$ with property A.  
If there is a point $x_0 \in X$ whose stabilizer in $G$ has property A, then $G$ has property A.
\end{cor}
\begin{proof}
The orbit $G x_0$ of $x_0$ in $X$ is coarsely equivalent to $X$, so $G x_0$ has property A.
By Theorem \ref{thm:PropAComplexRelPropA}, $G$ has relative property A with respect to the stabilizer.
Corollary \ref{cor:RelativePropAGivesPropA} gives property A for $G$.
\end{proof}

\begin{cor}
Suppose the finitely generated group $G$ is the fundamental group of a developable finite dimensional complex of groups
$\mathcal{Y}$, whose development is a locally finite complex with property A.  Then $G$ has relative property A with
respect to the vertex groups.  In particular, if each vertex group has property A, then $G$ has property A. 
\end{cor}

This also gives an extension of a Theorem of \cite{BCGNW}, where it is shown that a finite dimensional
CAT($0$) cube complex has property A.  In particular we have the following.

\begin{cor}
Suppose the finitely generated group $G$ acts cocompactly on a finite dimensional CAT($0$) cube complex.
If each vertex stabilizer has property A, then $G$ has property A.
\end{cor}

By the work of Campbell, \cite{C}, affine buildings have property A.  Thus we have the following which can be considered a
generalization of Kasparov-Skandalis, \cite{KS}.
\begin{cor}
Suppose the finitely generated group $G$ acts cocompactly on an affine building.  If the vertex stabilizers
have property A, then $G$ has property A.
\end{cor}

\section{Relative Amenability}
As Yu's property A is a generalization of amenability, the work of the proceeding sections specializes
to the notion of  relative amenability.  It is well-known that a countable discrete group $G$ is amenable
if and only if there exists a sequence of probability measures $\mu_n \in \Prob(G)$ such that for all
$g \in G$, $\| g \mu_n - \mu_n \|_{\ell^1} \to 0$.  That is, if $G$ acts amenably on a point.  This motivates
the definition of relative amenability from that of Definition \ref{def:BNNWRelAmenabAction}.

\begin{defn}\label{defn:RelAmenability}
A countable group $G$ is relatively amenable with respect to $\mH$ if $G$ acts amenable relative to $\mH$ on
a point.
\end{defn}

When $X$ is reduced to a point, much of the earlier notation simplifies.  Of particular interest are
$W_{0}(K, X) \isom \ell^1(K)$ and $N_{0}(K,X) \isom \ell^1_0(K)$, where $\ell^1_0(K)$ denotes the kernel
of the augmentation $\epsilon: \ell^1(K) \to \C$ given by $\epsilon(f) = \sum_{k \in K} f(k)$.

We remark that this definition, due to the above sections, is equivalent to the existence of a $\mu \in (\ell^1(K))^{**}$
with $g\mu = \mu$ and $\mu( \pi ) = 1$, where $\pi : \ell^1(K) \to \C$ is the augmentation map.  That is, rather
than the existence of a $G$-invariant mean on $\ell^\infty(G)$, relative amenability is the existence of 
a $G$-invariant mean on $\ell^\infty(K)$.

The construction of the relative Johnson class $[J_r] \in H^1_b\left(G; \left(\ell^1_0(K)\right)^{**} \right)$
is as before.  We have the following immediately from Theorem \ref{thm:RelJohnsonClass}.
\begin{prop}
The class $[J_r] \in H^1_b\left(G; \left(\ell^1_0(K)\right)^{**} \right)$ is trivial if and only if $G$ is relatively
amenable with respect to $\mH$.
\end{prop}

We note that when $X$ is a point, the notion of an $\ell^1$-geometric $G$-$C(X)$ module reduces
to just that of a $G$-Banach space.  This is due to the lack of a pair of nontrivial disjointly
supported elements of $C(X) \isom \C$.  
We obtain the following cohomological characterization of relative amenability.
\begin{thm}\label{thm:RelAmenabCohomChar}
Let $G$ be a countable group and $\mH$ a finite family of subgroups.  The following are equivalent.
\begin{enumerate}
	\item $G$ is relatively amenable with respect to $\mH$.
	\item For every $G$-Banach space $E$, the map $H^0_b( \mH; E^* ) \to H^1_b( G, \mH; E^* )$ is surjective.
\end{enumerate}
\end{thm}

This has the following corollary.  It follows as above, noting Johnson's characterization of
amenability.
\begin{cor}\label{cor:RelAmenabAndAmenabGivesAmenab}
Suppose the countable group $G$ is relatively amenable with respect to the family of subgroups $\mH$.
If each $H \in \mH$ is amenable, then $G$ is amenable.
\end{cor}

As an example, let $\mathbb{F}_2$ be the free group on two generators, $a$ and $b$, and let $A = <a>$ and $B = <b>$.
Since $A$ and $B$ are both amenable, if $\mathbb{F}_2$ were relatively amenable with respect to $\mH = \{ A, B \}$,
then $\mathbb{F}_2$ would be forced by Corollary \ref{cor:RelAmenabAndAmenabGivesAmenab} to be amenable.
As it is not, $\mathbb{F}_2$ is not relatively amenable with respect to $\mH$, even though it does have
relative property A with respect to $\mH$.

The following is stronger than the analogue for relative property A.
\begin{prop}
Suppose $H \normal G$.  The quotient $G/H$ is amenable if and only if $G$ is relatively amenable with respect to $H$.
\end{prop}
\begin{proof}
If $G/H$ is amenable, there is a sequence of probability measures $\mu_n \in \Prob(G/H)$ with
$\lim_{n \to \infty} \| qH \mu_n - \mu_n \|_{\ell^1} = 0$ for all $qH \in G/H$.  As the $G$ action
on $G/H$ has the property that $g qH = gH qH$ we have that 
$\| g \mu_n - \mu_n \|_{\ell^1} = \| gH \mu_n - \mu_n \|_{\ell^1}$.  In particular, for each $g \in G$
$\lim_{n\to\infty}\| g \mu_n - \mu_n \|_{\ell^1} = \lim_{n\to\infty} \| gH \mu_n - \mu_n \|_{\ell^1} = 0$.

For the converse, suppose $G$ is relatively amenable with respect to $H$.  Then there is a sequence
of probability measures $\xi_n \in \Prob(G/H)$, such that for all $g \in G$,
$\| g \xi_n - \xi_n \|_{\ell^1} \to 0$.  As $g qH = gH qH$ for all $qH \in G/H$, thus
for all $qH \in G/H$, $\| qH \xi_n - \xi_n \|_{\ell^1} \to 0$.  
\end{proof}

For instance, $\mathbb{F}_2 \oplus \Z^n$ is relatively amenable with respect to the $\mathbb{F}_2$,
but not with respect to the $\Z^n$.

\begin{prop}
Suppose $\mH$ is a finite family of finite index subgroups of the countable group $G$.
Then $G$ is relatively amenable with respect to $\mH$.
\end{prop}
\begin{proof}
Let $K = G/\mH$, and let $p$ be the uniform probability measure on the finite set $K$.
Let $\mu_n = p$ be the constant sequence.  As $p$ is invariant under the $G$ action
on $K$, $g \mu_n = \mu_n$.
\end{proof}

We recall the notion of a metric space being amenable, as in \cite{BW}.
For $(X, d_X)$ a metric space and $U \subset X$ let
$\partial_r U = \{ x \in X \, | \, d_X(x,U) < r \,\,\rm{ and }\,\, d_X(x, X \setminus U )< r \}$.
\begin{defn}
A uniformly discrete metric space with bounded geometry $(X, d_X)$ is amenable if for any 
$r, \delta > 0$ there is a finite $U \subset X$ so that 
\[ \frac{|\partial_r U|}{| U |} < \delta. \]
\end{defn}

Recall the definition of uniformly finite homology of Block-Weinberger.  Denote by
$C^{uf}_k(X)$ the formal sums $\sum_{z \in X^{k+1} } a_z [z]$, with $a_z \in \C$ satisfying
the following conditions.
\begin{enumerate}
	\item There is a $K > 0$ such that for all $z \in X^{k+1}$, $a_z < K$.
	\item There is an $R > 0$ such that if $z = (x_0, \ldots, x_k)$ with $d_X( x_i, x_j ) \geq R$
			then $a_z = 0$.
\end{enumerate}
Endowed with the boundary map, $\partial$,  induced by $(x_0, \ldots, x_k) \mapsto \sum_{j = 0}^{k} (-1)^j (x_0, \ldots, \hat{x_j}, \ldots, x_k)$,
we obtain a chain complex $C^{uf}_*(X)$.  Denote the homology of this complex by $H^{uf}_*(X)$.
A main result of \cite{BW} is the following.
\begin{thm}[Block-Weinberger]
Let $X$ be a uniformly discrete bounded geometry metric space.  The following are equivalent.
\begin{enumerate}
	\item $X$ is non-amenable.
	\item $H^{uf}_0(X) = 0$.
	\item If $c = \sum_{x \in X} a_x [x] \in C^{uf}_0(X)$ with each $a_x > 0$, then $[c] \neq 0$ in $H^{uf}_0(X)$.
\end{enumerate}
\end{thm}

\begin{thm}\label{thm:RelAmenActionOnSpace}
Suppose the countable discrete group $G$ acts cocompactly  on a uniformly discrete metric
space with bounded geometry $(X, d_X)$.  Let $\mathcal{R}$ be a family of representatives of
orbits of $G$ in $X$, and let $\mH$ be the family of subgroups which each stabilize an element
of $\mathcal{R}$.  If $X$ is an amenable metric space, then $G$ is amenable with respect to $\mH$.
\end{thm}
\begin{proof}
If $X$ is amenable, then $H^{uf}_0(X) \neq 0$.  We will make the identification $C^{uf}_0(X) \isom \ell^{\infty}(X)$. Let $\mathcal{A} =
\{ \phi \in \ell^\infty(X) \, | \, \exists K>0 \,\, \rm{ with } \,\ \phi(x) \geq K \, \forall x \in X \}$.
As in the proof of \cite[Theorem 3.1]{BW}, $\mathcal{A} \cap \partial C^{uf}_1(X)$ is empty and $\mathcal{A}$ is an
open convex subset of $\ell^\infty(X)$.  By the Hahn-Banach theorem there exists an $m \in (\ell^\infty(X))^* \isom (\ell^1(X))^{**}$ of norm one,
so that $m(\phi) \geq 0$ for all $\phi \in \mathcal{A}$, $m( \partial C^{uf}_1(X) ) = 0$, and $m( \phi_0 ) = 1$ where 
$\phi_0 = \sum_{x \in X} [x]$.  

For $\phi = \sum_{x \in X} a_x [x] \in \ell^\infty(G)$ and $g \in G$, define $\psi = \sum_{x \in X} a_x [x, gx]$.
As there is a constant $C$ such that $d_X( x, gx ) \leq C \ell_G(g)$, we have $\psi \in C^{uf}_1(X)$.
Moreover $\partial [x, gx] = [gx] - [x]$ so $\partial \psi = g \phi - \phi$. Thus for all $\phi \in \ell^\infty(X)$
and all $g \in G$, $g \phi - \phi \in \partial C^{uf}_1(X)$.  Therefore $m( g\phi - \phi ) = 0$.
As $g\phi$ and $\phi$ are both in $\ell^\infty(X)$, we have $(g m)( \phi) = m(\phi)$.
That is, $m \in (\ell^1(X))^{**}$ such that for all $g \in G$, $g m = m$ and $m( \phi_0 ) = 1$.
As $\phi_0 \in \ell^\infty(X)$ corresponds to $\pi \in (\ell^1(X))^*$ under the identification 
$(\ell^1(X))^* \isom \ell^\infty(X)$, we have the result.
\end{proof}

The following corollary is now clear.
\begin{cor}
Suppose the countable group $G$ acts cocompactly on a uniformly discrete metric space $X$ with bounded geometry.
If the space is amenable and there is a point $x_0 \in X$ with amenable stabilizer, then $G$ is amenable.
\end{cor}
\begin{proof}
This follows as in Corollary \ref{cor:ActionPropA}, noting that amenability of metric spaces
is a coarse invariant.
\end{proof}

We recall the definition of a co-amenable subgroup, as defined by Monod and Popa in \cite{MoP}.  
\begin{defn}
A subgroup $H$ of a group $G$ is called co-amenable in $G$ if every continuous affine $G$-action on a compact convex
subset of a locally convex space with an $H$-fixed point has a $G$ fixed point.
\end{defn}
This notion of co-amenability is equivalent to the existence of a $G$-invariant mean on $\ell^{\infty}(G/H)$.
It follows that co-amenability is equivalent to relative amenability in the case of a single subgroup.  As such, the
notion of relative amenability serves to generalize co-amenability.

We have the following generalization of Proposition 3 of \cite{MoP}, which extends Theorem \ref{thm:RelAmenabCohomChar}
above.
\begin{prop}\label{prop:higherDegrees}
Let $G$ be a countable group and $\mH$ a finite family of subgroups.  $G$ is relatively amenable with respect to $\mH$ if and only if
for every $G$-Banach space $E$ and each $n \geq 0$, the map $H^n_b( \mH; E^* ) \to H^{n+1}_b( G, \mH; E^* )$ is surjective.
\end{prop}
\begin{proof}
If $H^n_b( \mH; E^* ) \to H^{n+1}_b(G, \mH; E^* )$ is surjective for $n = 0$, then relative amenability follows from
Theorem \ref{thm:RelAmenabCohomChar}.  

The long exact sequence in relative bounded cohomology shows that the map $H^n_b( \mH; E^* ) \to H^{n+1}_b( G, \mH; E^* )$ is surjective
if and only if the map $H^{n+1}_b( G; E^* ) \to H^{n+1}_b( \mH; E^* )$ is injective.  By Proposition 10.3.2 and Lemma 10.3.6 of \cite{MoBook}, 
there is a Banach $G$-module $M$ such that $H^{n+1}_b(G; E^*) \isom H^{1}_b(G;M^*)$.  That a relatively injective $G$-module
can realize a relatively injective $H$-module for any subgroup $H < G$, also gives $H^{n+1}_b( \mH; E^* ) \isom H^1_b(\mH; M^* )$.
The surjectivity of $H^0_b( \mH; M^* ) \to H^{1}_b( G, \mH; M^* )$ completes the result.
\end{proof}

Bekka's notion of the amenability of unitary representations is also related to our notion of relative amenability \cite{Be}.
\begin{defn}
A unitary representation $\pi : G \to \mathcal{B}( \mathfrak{H})$ on a Hilbert space $\mathfrak{H}$ is amenable 
if there exists a state $\phi : \mathcal{B}( \mathfrak{H} ) \to \C$ such that for all $T \in \mathfrak{H}$ and all
$g \in G$, $\phi( \pi(g) T \pi(g)^* ) = \phi( T )$.
\end{defn}
The following is proved as in \cite{Be}.
\begin{prop}
Let $G$ be a countable group and $\mH$ a finite family of subgroups.  $G$ is relatively amenable with respect to $\mH$ if and only if
the quasi-regular representation $\pi : G \to \ell^2( G/\mH )$ is amenable.
\end{prop}

\end{document}